\documentclass[draft]{amsart}

\usepackage{amssymb}
\usepackage{url} 
\usepackage[usenames]{color}
\usepackage{amsmath,amsfonts,amsthm}
\usepackage{setspace}
\usepackage{hyperref}
\hypersetup{colorlinks=true}
\usepackage{syntonly}
\usepackage[mathcal]{euscript}
\usepackage{xypic}
\usepackage{color}   
  
\def\NN{{\mathbb N}}

\newcommand{\FDer}[1]{\stackrel{#1}{\to}}

\newcommand{\Fil}[1]{$C$-\textrm{filtered} }
\newcommand{\cal}{\mathcal}

\def\cT{{\cal T}}

\def\cV{{\cal{V}}}

\def\cK{{\cal K}}

\def\cV{{\cal V}}

\def\Der{\hbox{{\rm{Der}}}}

\def\Ker{\hbox{{\rm{Ker}}}}

\def\Ht{\hbox{{\rm{ht}}}}

\def\Ext{{\hbox{\rm{Ext}}}}
\def\Ass{{\hbox{\rm {Ass}}}}
\def\Hom{{\hbox{\rm{Hom}}}}
\def\Dim{{\hbox{\rm{dim}}}}
\def\Coker{{\hbox{\rm{Coker}}}}
\def\Im{{\hbox{\rm{Im}}}}

\def\Ht{{\hbox{ht}}}
\def\Ann{{\hbox{\rm{Ann}}}}
\def\InjDim{{\hbox{\rm{inj.dim}}}}
\def\Supp{{\hbox{\rm{Supp}}}}

\def\Spec{{\hbox{\rm{Spec}}}}

\def\Length{{\hbox{\rm{length}}}}

\newtheorem{Teo}{Theorem}[section]
\newtheorem{Lemma}[Teo]{Lemma}
\newtheorem{Prop}[Teo]{Proposition}
\newtheorem{Cor}[Teo]{Corollary}

\theoremstyle{definition}
\newtheorem{Def}[Teo]{Definition}

\newtheorem{Ex}[Teo]{Example}

\begin{document}

\title[Local cohomology modules of polynomial or power series rings]{Local cohomology modules of polynomial or power series rings over rings of small dimension}
\author{Luis N\'u\~nez-Betancourt}
\address{
Department of Mathematics,
University of Michigan, 
Ann Arbor, MI 48109, USA.
}
\email{luisnub@umich.edu}

\begin{abstract}
Let $A$ be a ring and $R$ be a polynomial or a power series ring over $A$. 
When $A$ has dimension zero,
we show that the Bass numbers and the associated primes of the local cohomology modules over $R$ are finite.
Moreover, if $A$ has dimension one and $\pi$ is an nonzero divisor, then the same properties hold for prime ideals that contain $\pi.$
These results do not require that $A$ contains a field. 
As a consequence, we give a different proof for the finiteness properties of local cohomology over unramified regular local rings. 
In addition, we extend previous results on the injective dimension of local cohomology modules over certain regular rings of mixed characteristic.
\end{abstract}

\maketitle
\section{Introduction}\label{1}
Throughout this manuscript, $A$ and $R$ denote commutative Noetherian rings with unity such that $R$ is either a polynomial ring, 
$A[x_1,\ldots,x_n]$, or a power series ring, $A[[x_1,\ldots,x_n]]$. 
If $M$ is an $R$-module and $I\subset R$ is an ideal, we denote the $i$-th local
cohomology of $M$ with support in $I$ by $H^i_I (M)$. If $I$ is generated by the elements $f_1,\ldots, f_\ell \in R$, these cohomology
groups can be computed using the $\check{\mbox{C}}$ech complex,

$$
0\to M\to \oplus_j M_{f_j}\to \ldots \to M_{f_1 \cdots f_\ell} \to 0.
$$

The structure of these modules has been widely studied 
by several authors. Among the results obtained, one encounters the following finiteness properties for certain 
regular rings:

\begin{itemize}
\item[(1)] \quad The set of associated primes of $H^i_I(R)$ is finite,
\item[(2)] \quad the Bass numbers of $H^i_I(R)$ are finite, and
\item[(3)] \quad $\InjDim_R H^i_I (R)\leq \Dim\Supp_R H^i_I(R)$.
\end{itemize}

Huneke and Sharp showed those properties for regular local rings of characteristic $p>0$ \cite{Huneke}. Lyubeznik proved these finiteness properties
for regular local rings of equal characteristic zero and finitely generated regular algebras over a field of characteristic zero
\cite{LyuDMod}.
It is conjectured that these properties also hold for regular local rings of mixed characteristic \cite{LyuDMod}.
We point out that these properties are not necessarily true in general; see \cite{Moty,Anurag,S-S} for counterexamples
of ($1$) and \cite{Ha} for a counterexample of ($2$). If $R$ is not a zero dimensional Gorenstein ring,
$R=H^0_0(R)$ is a counterexample for ($3$).

Properties ($1$), ($2$) and ($3$) have been proved for a larger family of functors introduced by Lyubeznik  \cite{LyuDMod}.
If $Z \subset \Spec ( R)$ is a closed subset and $M$ is an $R$-module, we denote by 
$H^i_Z (M)$
the $i$-th local cohomology module of $M$ with support in $Z$. We note that
$H^i_Z (R)=H^i_I (R)$,
where $Z=\cV(I)=\{P\in\Spec(R): I\subset P\}$.
For two
closed subsets of $\Spec(R)$, $Z_1\subset Z_2$, there is a long exact sequence of functors
\begin{equation}\label{LC2}
\ldots\to H^i_{Z_1}\to H^i_{Z_2}\to H^i_{Z_1/Z_2}\to \ldots
\end{equation}

A \emph{Lyubeznik functor}, $\cT$, is any functor of the form $\cT =\cT_1\circ \dots\circ\cT_t$, where every functor $\cT_j$
is either $H_{Z_1},$ $H^i_{Z_1\setminus Z_2},$ or the kernel, image, or cokernel of some arrow in the previous long exact sequence
for closed subsets $Z_1,Z_2$ of $\Spec(R)$ such that $Z_2\subset Z_1$.  We point out that $H^i_{Z_1/Z_2}$ was previously introduced 
(cf. \cite{RD}).

The aim of this manuscript is to extend Properties ($1$) and ($2$) for certain rings that are not necessarily regular or that do not necessarily contain a field. Namely,

\begin{Teo} \label{TeoZero}
Let $A$ be a zero dimensional commutative  Noetherian ring.
Let $R$ be either $A[x_1,\ldots,x_n]$ or $A[[x_1,\ldots,x_n]]$.
Then, 
\begin{itemize}
\item $\Ass_R \cT(R)$ is finite for every functor $\cT$, and
\item the Bass numbers of  $\cT(R)$ are finite.
\end{itemize}
In particular, these properties hold for $H^i_I(R)$ for every ideal $I\subset R$ and every integer $i\in \NN$.  
\end{Teo}

\begin{Teo} \label{TeoOne}
Let $A$ be a  one-dimensional ring, and 
$R$ be either $A[x_1,\ldots,x_n]$ or $A[[x_1,\ldots,x_n]].$ 
Let $\pi\in A$ denote an element such that $\dim(A/\pi A)=0.$ 
Then, 
the set of associated primes over $R$ of $\cT (R)$ that contain 
$\pi$ is finite for every functor $\cT$. Moreover, if $A$ is Cohen-Macaulay and $\pi$ is a nonzero divisor,
then the Bass numbers of  $\cT(R)$, with respect to a prime ideal $P$ that contains $\pi$, are finite.
In particular, these properties hold for $H^i_I(R)$ for every ideal $I\subset R$ and every integer $i\in \NN$.  
\end{Teo}

We also extend Property ($3$) to regular rings that do not necessarily contain a field. Namely:
\begin{Teo}\label{TeoInj}
Let $(A,m,K)$ be a regular local Noetherian ring and let $R$ be either $A[x_1,\ldots, x_n]$ or $A[[x_1,\ldots, x_n]]$. 
Let $M$ be a  $D(R,A)$-module supported only at $mR$. Then, 
$$
\InjDim_R(M)\leq \dim(A)+\dim(\Supp_R M).
$$
In particular, 
$$\InjDim_R(H^j_\eta (\cT (R)))\leq \dim(A),$$
where $\eta=(m,x_1,\ldots,x_n)R$ and $\cT$ is a Lyubeznik functor.
Moreover, if $R=A[x_1,\ldots, x_n]$, then
$$
\InjDim_R(M)\leq \dim(A)+\dim(\Supp_R M).
$$
for every $D(R,A)$-module, $M$.
\end{Teo}

Example \ref{ExSharp} shows that the bound given in Theorem \ref{TeoInj} is sharp.
However, to the best of our knowledge,
it is not known if this bound is sharp for local cohomology modules $H^i_I(R)$ if $A$ does not contain a field, even when $\dim(A)=1$.

The manuscript is organized as follows. In Section \ref{2}, we give an overview of $D$-modules. Later, in Section \ref{3}, we study the 
associated primes of $D(R,A)$-modules in the subcategory $C(R,A),$ introduced by Lyubeznik in \cite{LyuFreeChar}. 
In Section \ref{4}, we deal with the Bass 
numbers of $D(R,A)$-modules in $C(R,A)$,  and we prove 
Theorem \ref{TeoZero} and Theorem \ref{TeoOne}. 
In Section \ref{5}, we give a new proof for finiteness properties 
of local cohomology modules over a regular local ring of unramified mixed characteristic (cf. Theorem $1$ in \cite{LyuUMC}).
Finally, in Section \ref{6}, 
we study the injective dimension of $D(R,A)$-modules over a polynomial or a power
series 
ring with coefficients over any regular ring, and we prove Theorem \ref{TeoInj}, 
which generalizes Theorem $5.1$(a) in \cite{Zhou}.

\section{D-modules}
\label{2}
Given two rings, $A$ and $R$ such that $A\subset R$, we denote by $D(R,A)$ the ring of $A$-linear differential operators of $R$. This is the subring of 
$\Hom_A(R,R)$ defined inductively as follows. The differential operators of order zero are 
the morphisms induced by multiplying by  elements in $R$.
An element $\theta \in \Hom_A(R,R)$ is a differential operator of order less than or equal to $k+1$ if $\theta\cdot r -r\cdot\theta$
is a differential operator of order less than or equal to $k$ for every $r\in R=\Hom_R(R,R)$.

We recall that if $M$ is a $D(R,A)$-module, then $M_f$ has the structure of a $D(R,A)$-module  such that, for every $f \in R$, the natural
morphism $M\to M_f$ is a morphism of $D(R,A)$-modules. 
As a result of this, since $R$ is a $D(R,A)$-module, 
$\cT(R)$ is also a $D(R,A)$-module (cf. Examples $2.1$ in \cite{LyuDMod}).

By Theorem $16.12.1$ in \cite{EGA},
if $R=A[[x_1,\ldots,x_n]]$, 
then $$D(R,A)=R\left[\frac{1}{t!} \frac{\partial^t }{\partial x_i ^t} \ | \ t \in \NN, 1 \leq i \leq n \right]\subset \Hom_A(R,R).$$ 
For every ideal $I\subset A,$ there is a natural surjection $$\rho :D(R,A)\to D(R/IR,A/IA).$$
Moreover, if $M$ is a $D(R,A)$-module, then $IM$ is a $D(R,A)$-submodule and the structure of $M/IM$ as a $D(R,A)$-module
is given by $\rho$, i.e. $\delta \cdot v=\rho(\delta) \cdot v$ for all $\delta\in D(R,A)$ and $v\in M/IM$.\\

As in Lyubeznik \cite{LyuFreeChar}, we denote by $C(R,A)$
the smallest subcategory of $D(R,A)$-modules that
contains $R_f$ for all $f\in R$ and that is closed under subobjects, 
extensions, and quotients. In particular, the kernel, image, and cokernel of a morphism of $D(R,A)$-modules
that belong to $C(R,A)$ are also objects in $C(R,A)$. We remark that if $M$ is an object in $C(R,A)$,
then $\cT(M)$ is also an object in this subcategory; in particular, $\cT(R)$ belongs to $C(R,A)$ 
(cf. Lemma $5$ in \cite{LyuFreeChar}).\\

We note that if $R_f$ has finite length in the category of
$D(R,A)$-modules for every $f \in R$ and  $M$ is an object of $C(R,A)$, then $M$
has finite length as a $D(R,A)$-module. As a consequence, $\cT(R)$ would also have finite length.
We recall that, if $A=K$ is a field and $R$ is either $K[x_1,\ldots,x_n]$ or $K[[x_1,\ldots,x_n]]$, then $R_f$ has finite length in the category of
$D(R,K)$-modules for every $f \in R$ \cite{LyuDMod, LyuFreeChar}.

\section{Associated primes of local cohomology}
\label{3}

\begin{Lemma}\label{LemmaAssFL}
Let $A$ and $R$ be  Noetherian rings such that $A\subset R$. Let $M$ be a $D(R,A)$-module of finite length. Then,  $\Ass_R M$ is finite.
\end{Lemma}
\begin{proof}
Suppose $M\neq 0$. Let $M_1=M$ and  $P_1$ be a maximal element in the set of the associated primes of $M_1$.
Then, $N_1=H^0_{P_1}(M_1)$ is a nonzero $D(R,A)$-submodule of $M_1,$ and it has only one associated
prime. Given $N_{j}$ and $M_j$, set $M_{j+1}=M_{j}/N_j$.  If $M_{j+1}\neq 0$, let $P_{j+1}$ be a maximal element in the set of the
associated primes of $M_{j+1}$. Then 
$N_{j+1}=H^0_{P_j} (M_{j+1})$ has only one associated
prime. If $M_{j+1}=0$, set $N_{j+1}=0$ and $P_{j+1}=0$. Since $M_1=M$ has finite length as a $D(R,A)$-module, there exist $\ell\in \NN$ 
such that $M_j=0$ for $j\geq \ell$ and then
$\Ass(M)\subset\{P_1,\dots,P_\ell\}$.
\end{proof}

\begin{Lemma}\label{LemmaFLDZ}
Let $A$ be a zero-dimensional Noetherian ring. Let $R$ be either $A[x_1,\ldots,x_n]$ or $A[[x_1,\ldots,x_n]]$. 
Then, $R_f$ has finite length as a $D(R,A)$-module for every $f\in R$.
\end{Lemma}
\begin{proof}
Since  $A$ has finite length as a $A$-module, there is 
a finite filtration of ideals
$0=N_0\subset N_1\subset \ldots \subset N_\ell=A$
 such that $N_{j+1}/N_j$ is isomorphic to a field. Then, we have an induced filtration of $D(R,A)$-modules,
$0=N_0 R_f\subset N_1 R_f\subset \ldots \subset N_\ell R_f=R_f$. It suffices to prove that $N_{j+1} R_f/N_j R_f$
has finite length for $j=1,\ldots, \ell$.
We note that $N_{j+1} R_f/N_j R_f$ is zero or
isomorphic to $ (R/m)_f$ for some maximal ideal $m\subset A$. 
Since $N_{j+1} R_f/N_j R_f$ has finite length as a $D(R/m R,A/m A)$-module, 
it has finite length as a $D(R,A)$-module, which concludes the proof. 
\end{proof}

\begin{Prop} \label{PropAssDZ}
Let $A$ be a zero-dimensional commutative  Noetherian ring.
Let $R$ be either $A[x_1,\ldots,x_n]$ or $A[[x_1,\ldots,x_n]]$.
Then, $\Ass_R M$ is finite for every object in  $M\in C(R,A)$; in particular, this holds for $\cT (R)$ for every functor $\cT$.  
\end{Prop}
\begin{proof}
By Lemma \ref{LemmaFLDZ}, $R_f$ has finite length in the category of
$D(R,A)$-modules for every $f \in R$. If  $M$ is an object of $C(R,A)$, then $M$
has finite length as a $D(R,A)$-module, because length is additive.
\end{proof}

\begin{Lemma}\label{LemmaFLDO}
Let $A$ be a one-dimensional ring,  $\pi\in A$ be an element such that $\dim(A/\pi A)=0,$ 
and $R$ be either $A[x_1,\ldots,x_n]$ or $A[[x_1,\ldots,x_n]]$. 
Then, $R_f/\pi R_f$ has finite length as a $D(R,A)$-module for every $f\in R$.
\end{Lemma}
\begin{proof}
The length of $R_f/\pi R_f$ as a $D(R,A)$-module or as a $D(R/\pi R, A/\pi A)$-module is the same. 
Since $A/\pi A$ has dimension zero and $R/\pi R$ is either $(A/\pi A)[x_1,\ldots,x_n]$ or $(A/\pi A) [[x_1,\ldots,x_n]]$, 
the result follows from  Lemma \ref{LemmaAssFL} and Lemma \ref{LemmaFLDZ}.
\end{proof}

\begin{Lemma}\label{LemmaAQFL}
Let $A$ be a one-dimensional ring,  $\pi\in A$ be an element such that $\dim(A/\pi A)=0,$ 
and $R$ be either $A[x_1,\ldots,x_n]$ or $A[[x_1,\ldots,x_n]]$. Let $\bar{A}$ and $\bar{R}$ denote $A/\pi A$ and $R/\pi R$
respectively. Let $M$ be a $D(R,A)$-module, such that
$\Ann_{M} (\pi)$  and $M \otimes_R \bar{R}$ are objects in $C(\bar{R},\bar{A})$. Then, 
$\Ann_{\cT(M)} (\pi)$  and $\cT(M) \otimes_R \bar{R}$ are objects in $C(\bar{R},\bar{A})$
for every functor $\cT$.
\end{Lemma}
\begin{proof}
We recall that $\cT$ has the form $\cT =\cT_1\circ \dots\circ\cT_t$, 
where every functor $\cT_j$
is either $H^i_{Z_1},$ $H^i_{Z_1\setminus Z_2},$ or the kernel, image, or cokernel of some arrow in the  long exact sequence
\begin{equation}\label{LES}
\ldots\FDer{\alpha_i} H^i_{Z_1}(M)\FDer{\beta_i} H^i_{Z_2}(M)\FDer{\gamma_i} H^i_{Z_1/Z_2}(M)\to \ldots
\end{equation}
for closed subsets $Z_1,Z_2$ of $\Spec(R)$ such that $Z_2\subset Z_1$.

It suffices to prove the claim for $t=1$ by an inductive argument. Suppose that $\cT=H_Z(-),$ where $Z=Z_1\setminus Z_2$ for closed subsets $Z_1,Z_2\subset \Spec(R)$ such that $Z_2\subset Z_1$. 
We note that $H^i_{Z} (-)=H^i_{Z_1}(-),$ if we  choose $Z_2= \emptyset$. 
The exact sequences
$$
0\to\Ann_{M}(\pi)\to M\FDer{\cdot \pi} \pi M\to 0,
$$
and
$$
0\to\pi M \to M\to M\otimes_R \bar{R}\to 0,
$$
induce two long exact sequences,
\begin{equation*}\label{SES1}
\ldots\to H^i_Z(\Ann_{M}(\pi))\FDer{\phi_i} H^i_Z(M)\FDer{\varphi_i} H^i_Z(\pi M)\to\ldots
\end{equation*}
and
$$
\ldots \to H^i_Z(\pi M) \FDer{\phi'_i} H^i_Z(M)\FDer{\varphi'_i} H^i_Z(M\otimes_R \bar{R})\to\ldots.
$$
Since the composition of $\phi'_i \circ\varphi_i$ is the multiplication by $\pi$ on $H^i_Z( M),$
we obtain the exact sequences
$$
0\to \Ker(\varphi_i)\to \Ann_{H^i_Z(M)} (\pi)\FDer{\varphi_i} \Ker(\phi'_i),
$$
and
$$
\Coker(\varphi_i)\FDer{\phi'_i} H^i_Z(M)\otimes_R \bar{R} \to \Coker(\phi'_i)\to 0.
$$
Then, $\Ann_{H^i_Z(M)}(\pi)$ and $H^i_Z(M)\otimes_R \bar{R}$ are objects in $C(\bar{R},\bar{A}),$
as
$\Ker(\varphi_i),$  $\Ker(\phi'_i),$
$\Coker(\varphi_i)$ and  $\Coker(\phi'_i)$ belong to $C(\bar{R},\bar{A})$ and this category is closed under sub-objects, 
extensions and quotients.

If $\cT$ is a kernel, image, or cokernel of a morphism in the long exact sequence (\ref{LES}), there exists 
an injection,  $\cT(M) \to H^{i_1}_{Z_{j_1}}( M)$, and a surjection $H^{i_2}_{Z_{j_2}}( M)\to \cT(M)$ for some $i_1,i_2\geq 0$
and $j_1,j_2 \in\{1,2\}$. Then,
\begin{equation*}
0\to\Ann_{\cT(M)}(\pi)\to \Ann_{H^{i_1}_{Z_{j_1}}(M)} (\pi)
\end{equation*}
and
\begin{equation*}
H^{i_2}_{Z_{j_2}}(M)\otimes_R \bar{R}\to \cT(M)\otimes_R \bar{R}\to 0
\end{equation*}
are exact. Therefore, $\Ann_{\cT(M)} (\pi)$ and $\cT(M)\otimes_R \bar{R}$ belong to $C(\bar{R},\bar{A}).$
\end{proof}

\begin{Prop} \label{PropAssDO}
Let $A$ be a one-dimensional ring,  $\pi\in A$ be an element such that $\dim(A/\pi A)=0$, and
 $R$ be either $A[x_1,\ldots,x_n]$ or $A[[x_1,\ldots,x_n]]$.
Then, the set of associated primes over $R$ of $\cT (R)$ that contain 
$\pi$ is finite for every functor $\cT$.
\end{Prop}
\begin{proof}
The set of associated primes of $\cT (R)$ that contain 
$\pi$ is $\Ass_R \Ann_{\cT(R)}  (\pi).$ Since
$\Ann_{\cT(R)} (\pi)$ is a $D(R,A)$-module of finite length by Lemma \ref{LemmaAQFL}, we have that
$\Ass_R \Ann_{\cT(R)}  (\pi)$ is finite by Lemma \ref{LemmaAssFL}.
\end{proof}

\begin{Cor} \label{CorAssPoly}
Let $A$ be a one-dimensional local ring, and $R$ be $A[x_1,\ldots,x_n]$.
Then, $\Ass_R \cT(R)$ is finite.
\end{Cor}
\begin{proof}
Let $\pi$ be a parameter for $A$.
Then, the set of associated primes over $R$ of  $\cT (R)$ that contain 
$\pi$ is finite by Proposition \ref{PropAssDO}. Since $R_\pi=A_\pi[x_1,\ldots,x_n]$ and $\Dim(A_\pi)=0$, the
set of associated  primes over $R$ of $\cT (R)$ that does not contain 
$\pi$, which is in correspondence with $\Ass_{R_\pi} \cT (R_\pi)$, is finite by Corollary \ref{PropAssDZ}.
\end{proof}

\begin{Cor} \label{CorAssSeries}
Let $(A,m,K)$ be a one-dimensional local domain, and let $R$ be $A[[x_1,\ldots,x_n]]$.
Then, $\Ass_R \cT(R)$ is finite.
\end{Cor}
\begin{proof}
Let $\pi$ be a parameter for $A$.
Then, the set of associated primes over $R$ of  $\cT (R)$ that contain 
$\pi$ is finite by Proposition \ref{PropAssDO}.
It remains to show that the set of the associated primes not containing
$\pi$ is finite.

We will proceed by cases.
If $A$ is a ring of characteristic $p>0,$ we have that $R_\pi$ is a regular ring by  Theorem $5.1.2$ in  \cite{Gro}
because $R_\pi$ is the fiber at the zero prime ideal of $A.$ Then, 
$\Ass_{R_\pi} \cT (R_\pi)$ is finite by Corollary $2.14$ in \cite{LyuFMod}.

If $A$ is not a ring of characteristic $p>0.$ We have again that $R_\pi$ is a regular ring by  Theorem $5.1.2$ in \cite{Gro}. 
Let $F=A_\pi$ be the fraction field of $A$ and $S=F\otimes_A R=R_\pi.$
Then, $F$ is a field of characteristic $0$ and $S$ is an $F$-algebra. We claim
that $S$ and $F$ satisfy the properties:
\begin{itemize}
\item[\rm{(i)}] $S$ is equidimensional of dimension $n$;
\item[\rm{(ii)}] every residual field with respect to a maximal ideal is an algebraic extension of $F$;
\item[\rm{(iii)}] there exist $F$-linear derivations $\partial_1,\ldots, \partial_n \in \Der_{F} (S)$ and elements
$z_1\ldots, z_n\in R$ such that $\partial_i a_j=1$ if $i=j$ and 
$0$ otherwise.
\end{itemize}
 
We will proceed following the ideas of Lyubeznik in \cite{LyuUMC}. 
Let $\eta\subset S$ be a maximal ideal and let $Q = \eta\cap R.$ 
Then $Q$ is a prime ideal
of $R$ not containing $\pi,$ and it is maximal among the ideals of $R$ not containing $\pi.$ 
By induction on $n,$
it suffices to show that if $P$ is a nonzero prime ideal of $R$ not containing $\pi,$ then
there exist elements $y_1,\ldots, y_n\in R$ such that $R=A[[y_1,\ldots,y_n]]$ and
$R/P$ is a finitely generated $R_{n-1}/  P\cap R_{n-1}$-module, where $R_{n-1} = A[[y_1,\ldots, y_{n-1}]].$
Then, the finiteness implies that $\Ht(P\cap R_{n-1})= \Ht(P) -1.$ In addition, the
prime ideal $P$ is maximal among all ideals of $R$ not containing $\pi$ if and
only if $ P \cap R_n$ is maximal among all ideals of $R_n$ not containing $\pi.$
In addition, $S/PS=(R/P)\otimes_A F$ is an algebraic extension of $F$ if and only if $F\otimes_A R_{n-1}/P\cap R_{n-1}$ is an algebraic extension of $F$.

Let $\bar{P}$ be the image of $P$ in $\bar{R}=R/mR
= k[[x_1,\ldots,x_n]].$ There exist
new variables $y_1,\ldots,y_n$ such that
$\bar{R}/\bar{P}$ is finite over $\bar{R}_{n-1}/\bar{P}\cap \bar{R}_{n-1},$ where $\bar{R}_{n-1}=K[[y_1,\ldots y_{n-1}]].$
Let $r_1, \ldots ,r_s \in \bar{R} /\bar{P}$ be a set of generators over $\bar{R}_{n-1}/\bar{P}\cap \bar{R}_{n-1}.$
Lifting these variables to $R,$
we get that $R=A[[y_1,\ldots,y_n]]$. 
For every
$f \in R/ P$ there exist a finite number of elements $g_{1,1},\ldots, g_{1,s}, v_{1,j} \in R_{n-1}$ and $h_{1,j} \in m$ with 
$$ f=g_{1,1}r_1+\ldots +g_{1,s}sr_s+\sum_j h_{1,j} v_{1,j}.$$
We can apply the same idea to $v_{i,j}$ inductively to obtain a finite number of elements
$g_{t,1},\ldots, g_{t,s}, v_{t,j} \in R_{n-1}$ and $h_{t,j} \in m^t$  such that 
$$f=\left(\sum^t_{k=1}\sum_{i} h_{k-1,i}g_{k,1} \right)r_1+\ldots + \left(\sum^t_{k=1} \sum_{i} h_{k-1,i}g_{k,s}\right)r_s+\sum_j h_{t,j} v_{t,j}.$$
Since $R/P$ is $m$-adically separated and complete, we can take $$G_\ell=\sum^\infty_{k=1}\sum_{i} h_{k-1,i}g_{k,\ell},$$
and so, $f=G_1r_1+\ldots +G_s r_s.$ 
This proves that $r_1,\ldots, r_s$ is a finite system of generators of $R/P$ as
an $R_{n-1}/P R_{n-1}$-module and concludes the proof of the claim that
$R\otimes_A F$ and $F$ satisfy properties (i) and  (ii). 
In addition, we have that $z_i=x_i$ and $\partial_i=\frac{\partial}{\partial x_i}$ satisfies (iii).
Then, we have that $\Ass_{R_\pi} \cT (R_\pi)$ is finite by using the results of $D$-modules in \cite{MeNa} as it was done in \cite{LyuDMod}. 
It is proven explicitly in Theorem $4.4$ in \cite{Nunez}.
\end{proof}


\section{Bass numbers}
\label{4}
\subsection{Facts about Bass numbers}
\begin{Lemma}\label{LemmaBassN}
Let $(R,m,K)$ be a local ring and $M$ be an $R$-module. Then, the following are equivalent:
\begin{itemize} 
\item[\rm{a)}] $\Dim_K\Ext^j_R(K,M)$ is finite for all $j\geq 0$,
\item[\rm{b)}] $\Length(\Ext^j_R(N,M))$ is finite for every finite length module $N$  and for all $j\geq 0$, and
\item[\rm{c)}] $\Length(\Ext^j_R(N,M))$ is finite for one  module of finite length $N$ and for all $j\geq 0$.
\end{itemize}
\end{Lemma}
\begin{proof}
This is Lemma $3.1$ in \cite{NunezDS}.




\end{proof}

\begin{Lemma}\label{LemmaBCM}
Let $(R,m,K)$ be a Noetherian Cohen-Macaulay ring and $\pi\in R$ be a nonzero divisor. 
Let $M$ be an $R$-module annihilated by $\pi$. Then,
$\Dim_K\Ext^{\ell}_{R}(K,M)$ is finite for all $j\in \NN$ if and only if
$\Dim_K\Ext^{\ell}_{R/\pi R} (K,M)$ is finite for all $\ell\in \NN$. 
\end{Lemma}
\begin{proof}
 Let $g_i\in R$, such that $\pi,g_1,\ldots,g_n $ form a system of parameters.
Let $J$ denote $(\pi,g_1,\ldots,g_n)R$.
Using the Koszul complex  to compute the free resolution of $R/J$ as an $R$-module and as an $R/\pi R$-module, we obtain that 

\begin{align*}
\Length(\Ext^\ell_{R}(R/J,M))& =
\Length( \Ext^\ell_{R/\pi R}(R/J,M))\\
&+\Length (\Ext^{\ell-1}_{R/\pi R}({R/J},M)).
\end{align*}
The result follows from Lemma \ref{LemmaBassN},
because $R/J$ has finite length.
\end{proof}

\begin{Lemma}\label{LemmaBassQ}
Let $R$ be a Cohen-Macaulay local ring, $M$ be an $R$-module and $\pi\in R$ be a nonzero divisor.
Let $\bar{R}$ denote $R/\pi R$. Suppose that 
$$\Dim_K\Ext^j_{\bar{R}} (K, \Ann_M (\pi)) \hbox{ and }\Dim_K\Ext^j_{\bar{R}}(K,M\otimes_R \bar{R})$$ are finite for all $j\in \NN$. Then,
$\Dim_K\Ext^j_R (K,M)$ is finite for all $j\in R$.
\end{Lemma}
\begin{proof}

We have, by Lemma \ref{LemmaBCM}, that both
$$\Dim_K\Ext^j_{R} (K, \Ann_M (\pi)) \hbox{ and }\Dim_K\Ext^j_{R}(K,M\otimes_R \bar{R})$$ are finite for all $j\in \NN$. From the short exact sequences 
$$
0\to \Ann_{M} (\pi)\to M \FDer{\pi}  \pi M\to 0$$
and
$$0\to \pi M \to M \to M\otimes_R \bar{R}\to 0, 
$$
we get two long exact sequences induced by Ext:
$$
\ldots \to \Ext^\ell_R(K,\Ann_{M} (\pi))\FDer{\alpha_{\ell}} \Ext^\ell_R(K,M) \FDer{\beta_\ell }  \Ext^\ell_R(K,\pi M)
\to\ldots,
$$
and
$$
\ldots \to \Ext^\ell_R(K,\pi M)\FDer{\gamma_\ell } \Ext^\ell_R(K,M) \FDer{\theta_\ell }  \Ext^\ell_R(K, M\otimes_R \bar{R})
\to \ldots.
$$
Since $\Im(\theta_\ell )$ injects into $\Ext^\ell_R(K, M\otimes_R \bar{R})$,  we have that $\dim_K\Im(\theta_\ell )$ is finite. 
Likewise,  $ \Coker(\beta_{\ell} )$ 
injects 
into $\Ext^{\ell+1}_R(K,\Ann_{M} (\pi)),$  and it has finite dimension over $K.$
We note that $$\Ext^\ell_R(K,\pi M)=\Ker(\beta_\ell)\oplus \Coker(\beta_\ell).$$
Since $$\gamma_\ell \circ\beta_\ell=
\Ext^{\ell}_R (K,M)\FDer{\pi} \Ext^{\ell}_R (K, M)$$
is the zero morphism for $\ell\in \NN$, 
we have that $\Im(\gamma_\ell )=\gamma_\ell (\Coker(\beta_\ell ))$.
Therefore, $\gamma_\ell(\Coker(\beta_\ell)) \to \Ext^{\ell}_R(K,M)\to \Im(\theta_\ell )\to 0$
is exact, and then $\Dim_K(\Ext^\ell_R(K,M))$ is finite.
\end{proof}

\subsection{Finiteness properties of Bass numbers of local cohomology modules}
\begin{Def}\label{DefFil}
Let $A$ be a zero dimensional Noetherian ring. Let $R$ be either $A[x_1,\ldots,x_n]$ or $A[[x_1,\ldots,x_n]]$. 
Let $M$ be an $D(R,A)$-module. We say that $M$ is $C$-\emph{filtered} if there exists a filtration 
$0=M_0\subset M_1\subset \ldots \subset M_\ell=M$ of $D(R,A)$-modules, such that $M_{i+1}/M_i$ is either zero or
\begin{itemize}
\item[\rm{(1)}] $M_{i+1}/M_i$ is annihilated by a maximal ideal $m_i\subset R$,
\item[\rm{(2)}] $M_{i+1}/M_i$ is an object in $C(R/m_iR, A/m_i)$, and
\item[\rm{(3)}] $M_{i+1}/M_i$ is a simple $D(R, A)$-module.
\end{itemize}
\end{Def}

\begin{Lemma}\label{LemmaFil}
Let $A$ be a zero dimensional Noetherian ring. Let $R$ be either $A[x_1,\ldots,x_n]$ or $A[[x_1,\ldots,x_n]]$. 
Let $M$ be an object in $C(R,A)$. Then, $M$ is $C$-filtered.
\end{Lemma}
\begin{proof}
We first prove the claim for $R_f$ for every $f\in R.$
Since  $A$ has finite length as an $A$-module, there is 
a finite filtration of ideals,
$0=N_0\subset N_1\subset \ldots \subset M_\ell=A,$
such that $M_{i+1}/M_i$ is isomorphic to a field, $K_i=A/m_i,$ where $m_i$ is a maximal ideal of $A$. 
Then, we have an induced filtration of $D(R,A)$-modules,
$0=N_0 R_f\subset N_1 R_f\subset \ldots \subset N_\ell R_f= R_f$. 
Thus, $N_{i+1} R_f/N_i R_f=R_f/m_i R_f$, which is an object in $C(R/m_iR, A/m_i)$. 
Then, there exists a filtration, 
$N_i=M_{i,1}\subset \ldots \subset M_{i,j_i} =N_{i+1},$ 
of objects in $C(R/m_iR, A/m_i)$, such that 
$M_{i,t+1}/M_{i,t}$ is a simple $D(R/m_iR, A/m_i)$-module. Therefore, 
$$0=M_{0,1}\subset\ldots \subset M_{0,j_1}\subset M_{1,1}\subset\ldots\subset M_{\ell,j_{\ell}}=R_f$$ 
is a filtration that makes $R_f$ a $C$-filtered module.
By the definition of $C(R,A)$, it suffices to show that if $0\to M'\FDer{\alpha} M\FDer{\beta} M'' \to 0$ is a short exact sequence of
objects in $C(R,A)$, then $M$ is $C$-filtered   
if and only if $M'$ and $M''$ are  $C$-filtered.  If $M$ is $C$-filtered with a filtration $M_i$, we define a filtration $M'_i$ in $M'$
by $M'_i=\alpha^{-1}(M_i)$. Similarly, we define a filtration $M''_i$ in $M''$
by $M''_i=\beta (M_i)$. Then, we have a short exact sequence of short exact sequences:\\

\begin{figure}[h]
$$
\hskip -15mm
\begin{array}{*{18}{c@{\,}}}

&&
0 &&
0&&
0&&\\

&&
\downarrow &&
\downarrow &&
\downarrow &\\

0 & 
\FDer{} &
M'_{i} &
\FDer{\alpha} &
M'_{i+1} &
\FDer{\beta} &
M'_{i+1}/M'_{i} &
\to &
0\\

&&
\downarrow &&
\downarrow &&
\downarrow &\\

0 & 
\FDer{} &
M_{i} &
\FDer{\alpha} &
M_{i+1} &
\FDer{\beta} &
M_{i+1}/M_{i} &
\to &
0\\

&&
\downarrow &&
\downarrow &&
\downarrow &\\

0 & 
\FDer{\alpha} &
M''_{i} &
\FDer{\beta} &
M''_{i+1} &
\to &
M''_{i+1}/M''_{i} &
\to &
0\\

&&
\downarrow &&
\downarrow &&
\downarrow &\\

&&
0 &&
0&&
0&&\\

\end{array}
$$ \label{Figure 1}
\end{figure}
Since $M_{i+1}/M_i$ is either zero or a simple $D(R,A)$-module, $M'_{i+1}/M'_i$ is either $M_{i+1}/M_i$ or zero. Likewise,
$M''_{i+1}/M''_i$ is either $M_{i+1}/M_i$ or zero. Thus, $M'_i$ and $M''_i$ satisfy parts ($1$), ($2$), and ($3$) in Definition \ref{DefFil}.

On the other hand, if $M'$ and $M''$ are $C$-filtered modules with filtrations $M'_0\subset\ldots\subset M'_{\ell'}$ and 
$M''_0\subset\ldots\subset M''_{\ell''}$, respectively, we define a filtration on $M$ by $M_i=\alpha(M'_i)$ for $i=0,\ldots \ell'$ and 
$M_i=\beta^{-1}(M'_{i-\ell'})$ for $i=\ell'+1,\ldots, \ell'+\ell''$. Since $M_{i+1}/M_i=M'_{i+1}/M'_i$ for $i=0,\ldots \ell'$ and 
$M_{i+1}/M_i=M''_{i+1-\ell'}/M''_{i-\ell'}$ for $i=\ell'+1,\ldots \ell'+\ell''$, $M_i$ 
satisfies parts ($1$), ($2$) and ($3$) in Definition \ref{DefFil}. Hence, every object in $C(R,A)$  is a $C$-filtered module.
\end{proof}

\begin{Prop}\label{PropBDZ}
Let $A$ be a zero-dimensional Noetherian ring. Let $R$ be either $A[x_1,\ldots,x_n]$ or $A[[x_1,\ldots,x_n]]$. 
Let $M$ be an object in $C(R,A)$. Then, all the Bass numbers of $M$ are finite.
In particular, this holds for $\cT(R)$ 
for every 
functor $\cT$.
\end{Prop}
\begin{proof}
We fix a prime ideal $P\subset R$ and denote $R_P/PR_P$ by $K_P$.
Since $M$ is a $C$-filtered module by Lemma \ref{LemmaFil}, we have a filtration $0=M_0\subset\ldots \subset M_\ell=M$ such that 
$M_{i+1}/M_i$ is annulled by a maximal ideal $m_i\subset R$, is an object in $C(R/m_iR, A/m_i)$, and is a simple $D(R, A)$-module.
From the short exact sequences $0\to M_j \to M_{j+1}\to M_{j+1}/M_j\to 0 $, we get long exact sequences
$$
\ldots\to \Ext^j_{R_P} (K_P, (M_i)_P) \to \Ext^j_{R_P} (K_P,(M_{i+1})_P)$$
$$\to \Ext^j_{R_P} (K_P,(M_{i+1}/M_i)_P)\to \Ext^{j+1}_{R_P} (K_P, (M_i)_P) \to\ldots. 
$$
Then, it suffices to show the claim for $M_{i+1}/M_{i}$ for $i=0,\ldots, \ell-1$.
We fix $i$ and denote $M_{i+1}/M_{i}$ by $N$.
Let $m\subset A$ be the maximal ideal such that  $m N=0$. If $m R\not\subset P$, 
then $N\otimes R_P=0$. We may assume that $mR\subset P$. Let $\bar{R}=R/mR.$ We note that $\bar{R}$ is a regular ring containing  $A/m$, a field.
Let $g_1,\ldots, g_d$ be a system of parameters for $R_P,$ and let $f_1,\ldots, f_d$ be
the class of $g_1,\ldots,g_d$ in $\bar{R}_P$. Since $A$ is a zero dimensional ring, we have that $f_1,\ldots, f_d$ is a system of parameters for $\bar{R}_P$.
Let $I=(g_1,\ldots, g_d)R_P$. Using the Koszul complex $\cK$, we obtain that,
$$
\Ext^i_{R_P} (R_P/I, N_P)=H^i(\Hom_{R_P}(\cK (\underline{g}), N_P))
$$ $$
=H^i(\Hom_{\bar{R}_P}(\cK (\underline{f}), N_P))
=\Ext^i_{\bar{R}_P} (\bar{R}_P/I\bar{R}_P, N_P),
$$
because $R_P$ and $\bar{R}_P$ are Cohen-Macaulay rings of the same dimension.
Using Lemma \ref{LemmaBassN} several times, we obtain that
\begin{align*}
\Length_{\bar{R}_P} \Ext^i_{\bar{R}_P} (K_P, N_P)< \infty 
& \Leftrightarrow 
\Length_{\bar{R}_P} \Ext^i_{\bar{R}_P} (\bar{R}_P/I\bar{R}_P, N_P) <\infty \\
&\Leftrightarrow 
\Length_{\bar{R}_P} H^i\Hom_{\bar{R}_P}(\cK (\underline{f}), N_P) <\infty \\
& \Leftrightarrow 
 \Length_{R_P} H^i\Hom_{R_P}(\cK(\underline{g}), N_P)< \infty \\
& \Leftrightarrow 
 \Length_{R_P} \Ext^i_{R_P} (R_P/I, N_P)< \infty \\
& \Leftrightarrow 
 \Length_{R_P} \Ext^i_{R_P} (K_P, N_P)< \infty \\
\end{align*}
Since $\Length_{\bar{R}_P} \Ext^i_{\bar{R}_P} (K_P, N_P)< \infty$ by Corollary $8$ in \cite{LyuFreeChar}, we have that
$ \Length_{R_P} \Ext^i_{R_P} (K_P, N_P)< \infty$. Hence, all the Bass numbers of $M$ are finite. 
\end{proof}
\begin{proof}[Proof of Theorem \ref{TeoZero}]
This is a consequence of Proposition \ref{PropAssDZ} and Proposition \ref{PropBDZ}.
\end{proof}
\begin{Prop}\label{PropBassDO}
Let $A$ be a Noetherian Cohen-Macaulay ring such that $\dim(A)=1$, 
and let $\pi\in A$ be a nonzero divisor. Let $R$ be either $A[x_1,\ldots,x_n]$ or $A[[x_1,\ldots,x_n]]$.
Then, all the Bass numbers of  $\cT(R)$, as an $R$-module, with respect to a prime ideal $P$ containing $\pi,$ are finite.
\end{Prop}

\begin{proof}
Let $\bar{R}$ and $\bar{A}$ denote $R/\pi R$ and $A/\pi A,$ respectively. 
We have that $\Ann_{\cT(R)} (\pi)$ and $\cT (R)\otimes \bar{R}$ 
are objects in $C(\bar{R},\bar{A})$ by Lemma \ref{LemmaAQFL}. 
Then, their Bass numbers, as $\bar{R}$-modules, with respect to  
$P$ are finite by Proposition \ref{PropBDZ}. Since $R_P$ and $\bar{R}_P$ are Cohen-Macaulay rings, we have that the Bass
numbers of $\cT(R)$ with respect to $P$ are finite by Lemma \ref{LemmaBCM} for every functor $\cT$.
\end{proof}
We claim that we cannot generalize Proposition \ref{PropBassDO} for Cohen-Macaulay rings of dimension higher than $3$. 
Let $A=K[[s,t,u,v]]/(us+vt)$, where $K$ is a field. This is the ring given by Hartshorne's example \cite{Ha}. 
Let $I=(s,t)A$. Hartshorne showed that
$\dim_K \Hom_A(K,H^2_I(A)) $ is not finite.

Let $R$ be either $A[x_1,\ldots,x_n]$ or $A[[x_1,\ldots,x_n]]$. 
Let $P=mR$ be the prime ideal generated by $m=(s,t,u,v)A,$ the maximal ideal in $A.$
Then,

$$\Ext^0_R(R/P,H^2_I(R))=\Hom_R(R/P,H^2_I(R))$$
$$=\Hom_A(K,H^2_I(A))\otimes_A R=\oplus R/mR,$$
where the direct sum in the last equality is infinite. Therefore,
$$\dim_{R_P/mR_P}\Ext^0_{R_P}(R_P/PR_P,H^2_I(R_P))$$ is not finite.
It is worth noticing that $H^2_I(A)$ is simple as $D(A,K)$-module \cite{Hsiao}.

The only case in which Proposition \ref{PropBassDO} could generalize is for Cohen-Macaulay rings of dimension $2.$
This would be helpful to study the local cohomology over $$\frac{V[[x,y,z_1,\ldots,z_n]]}{(\pi-xy)V[[x,y,z_1,\ldots,z_n]]}=\left(\frac{V[[x,y]]}{(\pi - xy)V[[x,y]]}\right)[[z_1,\ldots,z_n]],$$
where $(V,\pi V,K)$ is a complete DVR of mixed characteristic.
To the best of our knowledge, this is the simplest example of a regular local 
ring of ramified mixed characteristic where the finiteness of $\Ass_RH^i_I(R)$ is an open question.

\begin{Cor}\label{CorBassPoly}
Let $A$ be a one-dimensional local Cohen-Macaulay ring, and let $R$ be $A[x_1,\ldots,x_n]$.
Then, all the Bass numbers of  $\cT(R)$, as an $R$-module, are finite.
\end{Cor}
\begin{proof}
Let $\pi$ be a parameter for $A$.
Then, the Bass numbers of  $\cT (R)$ with respect to a prime ideal containing $\pi,$ are finite by Proposition 
\ref{PropBassDO}.
Since $R_\pi=A_\pi[x_1,\ldots,x_n]$ and $\Dim(A_\pi)=0$, 
the Bass numbers of  $\cT (R)$ with respect to a prime ideal that does not contain $\pi,$ 
are finite by Proposition \ref{PropBDZ}.
\end{proof}

\begin{Cor}\label{CorBassSeries}
Let $A$ be a one-dimensional local domain, and let $R$ be $A[[x_1,\ldots,x_n]]$.
Then, all the Bass numbers of  $\cT(R)$, as an $R$-module, are finite.
\end{Cor}
\begin{proof}
Let $\pi$ be a parameter for $A$.
Then, the Bass numbers of  $\cT (R)$ with respect to a prime ideal $P$ containing $\pi,$ are finite by Proposition 
\ref{PropBassDO}.

On the other hand, the Bass numbers of $\cT (R)$ with respect to prime ideals not containing $\pi,$
are in correspondence with the Bass numbers of $R_\pi.$
We have  that $R_\pi$ is a regular ring that contains a field, $A_\pi,$ by  Theorem $5.1.2$ in  \cite{Gro}
because $R_\pi$ is the fiber at the zero prime ideal of $A.$
Then the result follows from Theorem $2.1$ in \cite{Huneke} and Theorem $3.4$ in  \cite{LyuDMod}.
\end{proof}

\begin{proof}[Proof of Theorem \ref{TeoOne}]
This is a consequence of Proposition \ref{PropAssDO} and Proposition \ref{PropBassDO}.
\end{proof}

\section{Local cohomology of unramified regular rings}
\label{5}

As consequence of the results in Sections \ref{3} and \ref{4}, we are able to give a different proof for some parts of 
Theorem $1$ in \cite{LyuUMC}.

\begin{Teo}
Let $(R,m,K)$ be an unramified regular local ring and $p=Char(K)$. Then:
\begin{itemize}
\item[\rm{(i)}]  the Bass numbers of $\cT(R)$ are finite, and
\item[\rm{(ii)}] the set of associated primes of $\cT(R)$ that contain $p$ is finite
\end{itemize}
for every Lyubeznik functor $\cT$.
\end{Teo}
\begin{proof}
The finiteness of associated primes of $\cT(R)$ that contain $p$ is not
affected by completion with respect to the maximal ideal. 
Since completion of $R$ respect to $m$ is a power series 
ring over a complete DVR  of mixed characteristic, the result follows from Proposition
\ref{PropAssDO}.

In order to prove the finiteness of the Bass numbers,
we need to show that $\Dim_{R_P/PR_P}\Ext^j_{R_P} (R_P/PR_P,\cT(R_P) )$ is finite
for every prime ideal $P\subset R$. 
There are two cases: $p\in P$ or not. If $p \not\in P$ then $R_P$ has equal 
characteristic $0$ and the result follows from Theorem $3.4$ in \cite{LyuDMod}.
If $p\in P$, $R_P$ is an unramified regular local ring, and 
its completion of $R_P$ respect to the maximal ideal is a power series 
ring over a complete DVR  of mixed characteristic. Since 
the dimension of  $\Ext^\ell_{R_P} (R_P/PR_P,\cT(R_P) )$
 as $R_P/PR_P$-vector space is not affected by completion, the result follows from 
Corollary \ref{PropBassDO}.
\end{proof}

\section{Injective Dimension}
\label{6}

In this section, we recover and generalize some results of Zhou about injective dimension \cite{Zhou}. 
\begin{Lemma}\label{LemmaExt}
Let $(A,m,K)$ be a regular local ring and let $R$ denote either
$A[[x_1,\ldots, x_n]]$ or $A[x_1,\ldots, x_n]$. 
Let $P\subset R$ be a prime ideal containing $mR$ and let $K_P$ denote the field $R_P/PR_P$.
Let $M$ be a  $D(R,A)$-module. Then, $\Ext^\ell_R (K_P,M_P)=0$ for $\ell > \Dim(A)+\Dim(\Supp_R (M))$.
\end{Lemma}
\begin{proof}
The proof will be by induction on $d=\dim(A)$. 
If $d=0$, then $A=K$ is a field, and the proof follows from the first Theorem in \cite{LyuInjDim}.
We assume that the claim is true for $d-1$. 
Let $y_1,\ldots, y_d$ denote a minimal set of generator for $m$.
Let $\bar{A}=A/y_d A,$ $\bar{R}=R/y_d R=\bar{A}[[x_1,\dots,x_n]]$ and $\bar{P}=P\bar{R}$. 
 Let $\bar{y}_1,\ldots, \bar{y}_{d-1}$ be the class of $y_1,\ldots, y_{d-1}$ in $\bar{R}$. 
We note that $\bar{P}\subset \bar{R}$ is a prime ideal, which contains $m\bar{R}$.
Let $f_1,\ldots f_s\in P$ be such that $y_1,\ldots, y_{d-1},f_1,\ldots f_s$ form a minimal set of generator for the 
maximal ideal $PR_{P}$.
From the Koszul complex associated to  $y_1,\ldots, y_{d-1},f_1,\ldots f_s$  in $R_{P}$, we have that for every $\bar{R}_{P}$-module, $N$,
$$\Dim_{K_P}\Ext^\ell_{R_P}(K_P, N)=
\Dim_{K_P} \Ext^\ell_{\bar{R}_{P}}(K_P, N)+\Dim_{K_P} \Ext^{\ell-1}_{\bar{R}_{P}}(K_P, N ).$$
In this case, we have that $\Ann_M (y_d)$ and $M/y_d M$ are 
$D(\bar{R}, \bar{A}))$-modules. By the induction
hypothesis, 
$$\Ext^{\ell}_{\bar{R}_{\bar{P}}}(K_P,\Ann_{M_P} (y_d))=0 \hbox{ and }\Ext^{\ell}_{\bar{R}_{\bar{P}}}(K, M_P/y_d M_P)=0$$  for $\ell >d+\Dim(\Supp_R (M))-1=
\Dim(\bar{A})+\Dim(\Supp_{R} (M))$. Then, 
$$\Ext^{\ell}_{R_P}(K_P,\Ann_{M_P} (y_d))=0\hbox{ and }
\Ext^{\ell}_{R_P}(K, M_P/y_d M_P)=0$$  for $\ell >d+\Dim(\Supp_R (M)).$

From the short exact sequences 
$$
0\to \Ann_{M_P} (y_d)\to M_P \FDer{y_d}  y_d M_P\to 0$$
and
$$0\to y_d M_P \to M_P \to M_P\otimes_R \bar{R}\to 0, 
$$
we get two long exact sequences induced by Ext:

$$
\ldots \to \Ext^\ell_{R_P}(K_P,\Ann_{M_P} (y_d))
\to \Ext^\ell_{R_P}(K_P,M_P)$$ $$\FDer{\rho_\ell }  \Ext^\ell_{R_P}(K_P,y_d M_P)
\to\ldots
$$
and
$$
\ldots \to \Ext^\ell_{R_P}(K_P,y_d M_P)\FDer{\theta_\ell } \Ext^\ell_{R_P}(K_P,M_P)$$
$$ \FDer{\varrho_\ell }  \Ext^\ell_{R_P}(K_P, M_P\otimes_R \bar{R})
\to \ldots
$$

Thus, $ \rho_\ell$ is an isomorphism and $\theta_\ell$ is surjective for $\ell>d+\Dim(\Supp_R (M))$. Then, $\theta_\ell \circ\rho$
is surjective for $\ell>d+\Dim(\Supp_R (M))$.
Since $$\theta_\ell \circ\rho_\ell=
\Ext^j_{R_P} (K_P,M_P)\FDer{y_d} \Ext^j_{R_P} (K_P, M_P)$$ 
is the zero morphism,  $ \Ext^j_{R_P} (K_P, K_P)=0$ for $\ell >d+\Dim ( \Supp_R M)$.
\end{proof}

\begin{Prop}\label{PropInjDimPol}
Let $A$ be a Noetherian ring and let $R=A[x_1,\ldots, x_n]$. 
Let $P\subset R$ be a prime ideal, and let $K_P$ denote the field $R_P/PR_P$.
Let $M$ be a  $D(R,A)$-module. Then, $\Ext^\ell_R (K_P,M_P)=0$ for $\ell > \Dim(A)+\Dim(\Supp_R M)$.
\end{Prop}
\begin{proof}
Let $Q=P\cap A$.  Then, $PR_Q$ is a prime ideal in $R_Q$ that contains $QR_Q$. Since $R_Q=A_Q[x_1,\ldots,x_n],$ $M_Q$ is a
$D(R_Q,A_Q)$-module, and $(M_Q)_P=M_P$, we have that 
$\Ext^\ell_{R_P} (K_P,M_P)=0$ is zero for $\ell > \Dim(A)+\Dim(\Supp_R (M_P))$ by Lemma \ref{LemmaExt}. Hence,
$\Ext^\ell_{R_P} (K_P,M_P)=0$ is zero for $\ell > \Dim(A)+\Dim(\Supp_R (M)),$ because $\Dim(\Supp_R (M))\geq \Dim(\Supp_R (M_Q))$.
\end{proof}

\begin{proof}[Proof of Theorem \ref{TeoInj}] 
Let $0\to E^0\to E^1\to E^2\to\ldots$ be  the minimal free resolution of $M.$ We know that every injective module $E$ is isomorphic to a direct sum 
of $E(R/P),$ where $P$ is a prime ideal and  $E(R/P)$ denotes the injective hull of $R/P.$ 
The number of copies of $E(R/P)$ in $E^j$ is given by $\dim_{R_P/PR_P}\Ext^j_{R_P}(R_P/PR_P,M_P),$ the $j$-th Bass number
of $M$ with respect to $P.$ Therefore, the injective dimension of  $M$ is bounded by a number $B$ if and only if 
$\Ext^j_{R_P}(R_P/PR_P,M_P)=0$ for every integer $j> B$ and prime ideal $P\subset R.$

If $M$ is supported only at $mR,$ then 
$M_P=0$ for every prime ideal $P$ such that $mR\not\subset P.$ Therefore, it suffices to prove that $\Ext^\ell_R (K_P,M_P)=0$ for $\ell > \Dim(A)+\Dim(\Supp_R (M))$ and prime ideals that contain $mR,$ which follows from Lemma \ref{LemmaExt}.

We now assume that  $R=A[x_1,\ldots, x_n]$ and take $M$ to be any $D(R,S)$-module.
In this case, we can reduce the the situation in which $M$ is supported only at $mR$ to compute the Bass numbers as it was done in 
Proposition \ref{PropInjDimPol}.
Therefore,  $\Ext^\ell_R (K_P,M_P)=0$ for every $\ell > \Dim(A)+\Dim(\Supp_R (M))$ and  every prime
ideal $P\subset R$ by Proposition \ref{PropInjDimPol}. 
\end{proof}
Now we give an example, inspired by work of Zhou \cite{ZhouEx}, that shows that the bound presented in Theorem  \ref{TeoInj} is sharp.
\begin{Ex}\label{ExSharp}
Let $(A,m,K)$ be a regular local ring of dimension $d$.
Let $R$ be either $A[x_1,\ldots,x_n]$ or $A[[x_1,\ldots,x_n]]$.
Let $M=R/mR.$ Following the proof of Lemma \ref{LemmaExt}, it can be shown that
$$
\InjDim_R (M)=\dim(A)+\dim\Supp_R(M)=d+n.
$$
\end{Ex}

\section*{Acknowledgments}
I would like to thank Mel Hochster for his invaluable comments and suggestions.  
I also wish to thank Gennady Lyubeznik for his helpful comments; in particular, 
for suggesting a generalization of Lemma \ref{LemmaAQFL} to every functor $\cT$ instead of only for local cohomology modules.  
I would like to thank Wenliang Zhang for pointing out that the hypothesis of a Cohen-Macaulay ring, previously assumed, was not needed in several results; 
in particular, in Proposition \ref{PropAssDO}.
I also wish to thank him for suggesting the statement and the ideas behind the proof of Corollary \ref{CorAssSeries}.
I thank the referee for very helpful comments.
Thanks are also due to the National Council of Science and Technology of Mexico by its support through grant $210916.$

\bibliographystyle{alpha}
\bibliography{References}

\end{document}